\documentclass[11pt]{article}
\usepackage{amsmath,amsfonts,amsthm,amssymb, graphicx,color}

\newtheorem{theorem}[equation]{Theorem}
\newtheorem{prop}[equation]{Proposition}
\newtheorem{lemma}[equation]{Lemma}

\newtheorem{remark}[equation]{Remark}

\numberwithin{equation}{section}

\renewcommand{\qed}{\hspace*{\fill} \setlength{\unitlength}{1mm}
\begin{picture}(2.5,2.5)
      \put(0,0){\framebox(2.5,2.5){}}
\end{picture}
\setlength{\unitlength}{1pt}}

\newcommand{\Tr}{\mathrm{Tr}\,}

\newcommand{\R}{\mathbb{R}}

\newcommand{\RR}{\mathbb{R}}
\newcommand{\calC}{\mathcal C}
\newcommand{\calG}{\mathcal G}

\newcommand{\del}{\partial}
\newcommand{\e}{\epsilon}
\newcommand{\calU}{\mathcal U}
\newcommand{\calO}{\mathcal O}
\newcommand{\calV}{\mathcal V}

\title{A heat trace anomaly on polygons}
\author{Rafe Mazzeo \thanks{Supported by NSF grant 
DMS-0805529}\\ Stanford University \and Julie Rowlett \thanks{Current address: Mathematisches Institut, University of Bonn}
\\ UC Santa Barbara}

\date{}

\begin{document}
\maketitle


\begin{abstract}
Let $\Omega_0$ be a polygon in $\RR^2$, or more generally a compact surface with piecewise smooth 
boundary and corners. Suppose that $\Omega_\e$ is a family of surfaces with $\calC^\infty$ boundary which 
converges to $\Omega_0$ smoothly away from the corners, and in a precise way at the vertices to be described in
the paper. Fedosov \cite{Fe}, Kac \cite{K} and McKean-Singer \cite{MS} recognized that certain 
heat trace coefficients, in particular the coefficient of $t^0$, are not continuous as $\e \searrow 0$. 
We describe this anomaly using renormalized heat invariants of an auxiliary smooth domain $Z$ which
models the corner formation. The result applies both for Dirichlet and Neumann conditions. We also include
a discussion of what one might expect in higher dimensions. 
\end{abstract}

\section{Introduction}
Let $\Omega \subset \RR^2$ be a domain with smooth boundary, or more generally, any two dimensional
compact Riemannian manifold with smooth boundary. The Laplace operator with Dirichlet 
boundary conditions has discrete spectrum $\{\lambda_i\}$ and corresponding eigenfunctions $\{\phi_i\}$. 
The fundamental solution for the Cauchy problem for the heat equation has Schwartz kernel
\[
H^\Omega(t,z,z') = \sum_{i=1}^\infty e^{-\lambda_i t}\phi_i(z)\phi_i(z');
\]
this converges in $\calC^\infty((0,\infty) \times \overline{\Omega} \times \overline{\Omega})$ and is
smooth up to $t=0$ away from the diagonal of $\Omega \times \Omega$. The so-called heat trace is the function
\begin{equation}
\Tr H^\Omega = \sum_{i=1}^\infty e^{-\lambda_i t} = \int_{\Omega} H^\Omega(t,z,z)\, dz;
\label{eq:heattrace}
\end{equation}
this has an asymptotic expansion as $t \searrow 0$ of the form
\begin{equation}
\mbox{Tr\,} H^\Omega \sim \sum_{j=0}^\infty a_j t^{-1 + \frac{j}{2}}.
\label{eq:heattrace2}
\end{equation}
Each coefficient $a_j$ is a sum of two terms: an integral over $\Omega$ of some universal polynomial 
in the Gauss curvature $K$ of the metric and its covariant derivatives, and an integral over 
$\del \Omega$ of another universal polynomial in the geodesic curvature $\kappa$ of the boundary and its 
derivatives. Precise formul\ae\ for these polynomials are extremely complicated (and mostly unknown) when
$j$ is large, but the first few are quite simple: 
\[
a_0 = \frac{1}{4\pi}\int_\Omega 1\, dA = \frac{1}{4\pi} |\Omega|,\qquad 
a_1 = -\frac{1}{8\sqrt{\pi}}\int_{\del \Omega} 1\, ds = -\frac{1}{8\sqrt{\pi}} |\del \Omega|
\]
and
\begin{equation} \label{eq:a2smooth}
a_2 = \frac{1}{12\pi}\left(\int_\Omega K\, dA + \int_{\del \Omega} \kappa\, ds\right) = \frac{1}{6}\chi(M).
\end{equation}
Here and elsewhere, $| \cdot |$ refers to either area of a domain or length of its boundary, as 
appropriate.

Almost all of this remains true if the boundary of $\Omega$ is piecewise smooth. More precisely,
assume that $\del \Omega$ is a finite union of smooth arcs, $\gamma_i$, $i = 1, \ldots, k$, where 
(counting indices mod $k$) $\gamma_i$ meets $\gamma_{i+1}$ at the vertex $p_i$ with an interior angle
$\alpha_i \in (0,2\pi)$. In fact, the only modification in the statements above is that the
heat trace coefficients may now include contributions from the vertices. The formul\ae\ for $a_0$
and $a_1$ are the same as before, but now  
\begin{equation}
a_2 = \frac{1}{12\pi}\left(\int_{\Omega} K\, dA + \sum_{j=1}^k \int_{\gamma_j} \kappa\, ds\right)
+ \sum_{j=1}^k \frac{\pi^2 - \alpha_j^2}{24\pi \alpha_j}.
\label{eq:a2}
\end{equation}
The term in parentheses on the right now equals $2\pi \chi(\Omega) - \sum_{j=1}^k(\pi-\alpha_j)$.
That the coefficient $a_2$ contains an extra contribution from the vertices was already known to
Fedosov \cite{Fe} (who was studying Riesz means of the eigenvalues on polyhedra of 
arbitrary dimension) and to Kac \cite{K}, although the precise simple expression here was obtained 
by Dan Ray (this is referenced by Kac and also later by Cheeger \cite{Ch}, but apparently Ray 
did not publish his result). A particularly transparent derivation of this corner term appears 
in a paper by van den Berg and Srisatkunarajah \cite{BS}.

The heat trace anomaly in the title of our paper is the discrepancy between the heat coefficients
in the smooth and polygonal settings. More specifically, it refers to the fact that at least one heat
invariant is not continuous with respect to Lipschitz convergence of domains. To phrase
this more precisely, let $\Omega_\e$ be a family of surfaces with {\it smooth} boundary which converge 
to a piecewise smoothly bounded domain $\Omega_0$ as $\e \to 0$. We think of $\Omega_\e$ as $\Omega_0$
with each corner `rounded out' slightly, but will give a precise formulation in the next paragraph. 
Denoting the heat trace coefficients for $\Omega_\e$ by $a_j(\e)$, it will be clear from this 
definition that
\begin{multline*}
\lim_{\e \to 0} a_2(\e) = \lim_{\e \to 0} \frac{1}{12 \pi}\left(\int_{\Omega_\e} K_\e\, dA_\e + 
\int_{\del \Omega_\e} \kappa_\e \, ds\right) \longrightarrow  \\
\frac{1}{12\pi} \left(\int_{\Omega_0} K_0\, dA_0 + \sum_{i=1}^k \int_{\gamma_i} \kappa_0 \, ds+ \sum_{i=1}^k (\pi-\alpha_i)\right),
\end{multline*}
where $K_\e$ and $\kappa_\e$ are the Gauss curvatures of $g_\e$ and the geodesic curvatures of $\del \Omega_\e$
for every $\e \geq 0$, respectively. The anomaly is simply that this formula does not agree with the expression 
(\ref{eq:a2}). The goal of this paper is to provide a simple explanation for the disagreement between these two 
expressions. 

We now explain the desingularization more precisely. For simplicity, suppose that $\Omega_0$ and $\Omega_\e$
all lie in some slightly larger ambient open surface $\widetilde{\Omega},$ and that the metrics $g_\e$ on $\Omega_\e$
are all extended to metrics (still denoted $g_\e$) on this larger domain. We assume that this family of
metrics converges smoothly on $\widetilde{\Omega}$. Let $p$ be a vertex of $\Omega_0$ and consider the portion 
of $\Omega_\e$ in some ball of fixed size around $p$, $B_c(p) \cap \Omega_\e$. Our main assumption is that the 
family of pointed spaces $(B_c(p) \cap \Omega_\e,\e^{-2}g_\e,p)$ converges in pointed Gromov-Hausdorff norm, and 
smoothly, to a noncompact region $Z \subset \RR^2$ with smooth boundary, such that at infinity, $\del Z$ is 
asymptotic to a cone with vertex at $0$ and with opening angle $\alpha$, the same angle as at the 
vertex $p$ in $(\Omega_0,g)$.  Note that this is actually pointed Gromov-Hausdorff convergence for the
ambient space $(\widetilde{\Omega},g_\e,p)$. 

Note that this definition implies that the distance between $p$ and $\del \Omega_\e$ is bounded above
by a constant times $\e$, and that $g_\e$ is a small perturbation, which decreases with $\e$, of the
rescaling of the standard flat metric on $Z \cap B_{c/\e}$. For convenience we assume in the rest of this
paper that the constant $c$ equals $1$. Thus the basic assumption is the existence of a smoothly
bounded asymptotically conic region $Z$ in the plane such that $\e^{-1}(\Omega_\e \cap B_1(p))$ 
converges to $Z$. 

This definition is a very special case of a more general desingularization construction explored carefully 
in \cite{R0} and \cite{R1} for the case of degeneration to spaces with isolated conic singularities, and
in greater generality in \cite{evian}. The goal in these first two papers, as here, is to analyze the 
behaviour of the heat kernel under this degeneration process. That analysis is quite involved, although 
it yields much sharper results than can be obtained by the present more naive methods. However, one 
motivation for the present paper is to show how some very simple rescaling arguments, which are only
slight generalizations of ones used (in substantially more sophisticated ways) by Cheeger \cite{Ch}, 
already yield some interesting results. 

Now consider the function
\begin{equation}
G(t,\e) = \Tr H^{\Omega_\e} = \int_{\Omega_\e} H^{\Omega_\e}(t,z,z)\, dz,
\label{de:F}
\end{equation}
which is smooth on the interior of the quadrant $Q = \{t \geq 0, \e_0 > \e\geq 0\}$; our main 
theorem concerns its precise regularity at the corner $t=\e = 0$. This will be decribed in terms of
its regularity on the parabolic blowup of $Q$ which we denote $Q_0$. This space is diffeomorphic
to $Q$ away from the origin, but has an extra `front face' $F$ replacing the point $(0,0)$ which
encodes all the directions of approach to this point along parabolic trajectories. It is described
more carefully in \S 2 below. One of the goals of this paper, in fact, is to advertise the utility 
and naturality of this blowup construction. 

\begin{theorem}
Let $(\Omega_\e,g_\e)$ be a family of smooth surfaces with Riemannian metrics which converge in the 
manner described above to a surface with piecewise smooth boundary $(\Omega_0,g_0)$. Then the 
function $G(t,\e)$ lifts to $Q_0$ to be polyhomogeneous conormal at all boundary faces and corners. 
\label{th:blowup}
\end{theorem}

Recall that polyhomogeneity means simply that the lift of $G$ has asymptotic expansions at all
boundary faces and product type expansions at all corners. The existence of such expansions 
somehow normalizes our problem. Indeed, the heat trace anomaly is simply the fact that the
limit as $\e \searrow 0$ of the second asymptotic coefficient $a_2(\e)$ in the expansion
as $t \searrow 0$ is not the same as the second asymptotic coefficient of the heat expansion
for $\Omega_0$. The front face $F$ of $Q_0$ separates where these limits are taken (first $t \to 0$ then
$\e \to 0$ vs.\ the other way around), and this extra space allows for the existence of a
function which interpolates between these two values. Our second main result describes this function.

\begin{theorem}
There is a function $C_2(\tau)$ defined along the front face of $Q_0$, which is smooth in the 
rescaled time variable $\tau = t/\e^2$, and satisfies
\begin{multline*}
\lim_{\tau \searrow 0} C_2(\tau) = \frac{\chi(\Omega_0)}{6}, \qquad \mbox{and} \\
\lim_{\tau \nearrow \infty} C_2(\tau) = \frac{\chi(\Omega_0)}{6} + 
\sum_{j=1}^k \frac{\pi^2 - \alpha_j^2}{24\pi \alpha_j} - \frac{1}{12\pi} \sum_{j=1}^k (\pi-\alpha_j).
\end{multline*}
Its explicit form includes the finite part of a divergent expansion:
\[
C_2(\tau) = \frac{\chi(\Omega_0)}{6} + 
\sum_{j=1} ^k \underset{{\e=0}}{\mathrm{f.p.}} \int_{\{z \in Z_j: |z|<1/\e\}} H^{Z_j}(\tau,z,z)\, dz - \frac{1}{12 \pi}
\sum_{j=1}^k (\pi-\alpha_j),
\]
where $Z_j$ is a noncompact region in the plane which models the collapse at the $j^{th}$ corner.  
\label{th:fp}
\end{theorem}

\begin{remark} When $\Omega_0$ is a triangle (or indeed, any polygon in the plane), 
the first and third terms in the formula for $\lim_{\tau \to \infty} C_2(\tau)$
cancel, and we obtain Ray's original formula
\[
\lim_{\tau \to \infty} C_2(\tau) = a_2(0) = \sum_{j=1}^k \frac{\pi^2 - \alpha_j^2}{24 \pi \alpha_j}.
\]
\end{remark}

This interpolating function $C_2(\tau)$ therefore `explains' the heat trace anomaly, or alternately, 
the anomaly is caused by the renormalized heat trace on the complete space $Z_j$. We also discuss some of
the other coefficients in the asymptotic expansions for the lift of $G$ at the various boundary faces 
and corners of $Q_0$. 

Finally, we note that the behaviour of spectral quantities under `self-similar smoothing of corners'
in two-dimensional domains has been considered elsewhere. In particular, Dauge, Tordeux and Vial \cite{DTV}
have carried out an extensive analysis of the asymptotic behaviour of solutions of $\Delta u = f$
on such a family of domains. 

This paper is organized as follows.  In \S 2, we recall some preliminary facts about parabolic blowups and
scaling properties of heat kernels and the standard parametrix construction for heat kernels. The proofs
of the two theorems are then presented in \S 3.  In \S 4 we indicate the minor modifications needed to
prove the analogous result for Neumann boundary conditions; the statement of the main theorem in
that setting will be given there. Finally, in \S 5 we include, at the referee's
request, a brief discussion of what can be done in higher dimensions. The results there are less explicit,
but the proofs carry over fairly directly. 

The authors wish to thank Lennie Friedlander for bringing this problem to their attention, and for making some
very useful remarks on the first version of this manuscript; the first
author is also grateful to Gilles Carron and Andrew Hassell for some helpful comments.  

\section{Preliminaries}
In this section we collect the requisite facts and tools: the behaviour of the heat kernel under scaling 
of the underlying space, a review of parabolic blowups and polyhomogeneity, and a slight modification of the 
standard parametrix construction for heat kernels. 

\subsection{Heat kernels and dilations}
The heat kernel transforms naturally under dilations of the domain, or equivalently, of the metric.  Let
$(M,g)$ be any complete Riemannian manifold with smooth (or piecewise smooth) boundary, and denote by 
$H^M(t,z,z')$ the minimal heat kernel for the Laplacian with Dirichlet boundary conditions on $M$. 
This is a smooth function on the interior of $\RR^+ \times M \times M$ with well-known regularity 
properties at the various boundaries and corners. 

We seek to relate this heat kernel with the one for the same manifold $M$ but with rescaled metric 
$g_\lambda = \lambda^2 g$, $\lambda \in \RR^+$. This will be applied when $M \subset \RR^2,$ $g$ is 
the induced Euclidean metric, and we relate its heat kernel to the one for $\lambda M$, the image of 
$M$ under the dilation $D_\lambda:  \RR^2 \to \RR^2$, $z \mapsto \lambda z$. The pullback of the Euclidean
metric from $\lambda M$ to $M$ is simply $\lambda^2\, g$. 

\begin{prop}
The heat kernels on $M$ and $\lambda M$ are related by the formula
\[
H^{\lambda M}(\lambda^2 t, \lambda z, \lambda z') \lambda^{2} = H^M(t,z,z').
\]
\label{pr:dilate}
\end{prop}
Implicit in this formula, we are parametrizing points in $\lambda M$ with points 
in $M$ via $D_\lambda$. To prove this, observe that the heat operator $\del_t - \Delta_z$
on $M$ transforms homogeneously with respect to the parabolic dilation $(t,z) \mapsto (\lambda^2 t, 
\lambda z)$. Hence, the expression on the left satisfies the heat equation; the additional $\lambda^{2}$
is the Jacobian factor accounting for the fact that $H^{\lambda M}(0,w,w') = \delta(w-w')$ 
is homogeneous of order $-2$ in two dimensions. 

\subsection{Parabolic blowup}
The parabolic dilation $D_\lambda(t,\e) = (\lambda^2 t, \lambda \e)$ motivates the introduction of 
parabolic blowup $Q_0$ of the quadrant $Q := [0,\infty)_t \times [0,\e_0)_\e$ at $(0,0)$. This space is 
defined as follows. As a set, $Q_0$ is the disjoint union of $Q \setminus \{(0,0)\}$ and the orbit 
space $F = \left(Q \setminus \{(0,0)\}\right)/\sim$, where $(t,\e) \sim (t',\e')$ if $(t',\e') = D_\lambda (t,\e)$ 
for some $\lambda > 0$. More concretely, $F$ is diffeomorphic to a closed quarter-circle; it is also 
identified with the set of all equivalence classes of parametrized curves $\gamma(s) = (t(s), \e(s))$ with 
$\lim_{s\searrow 0} \gamma(s) = (0,0)$, $\e(s) = \calO(t(s)^2)$, and where
\[
\gamma \sim \tilde{\gamma} \Longleftrightarrow \lim_{s \to 0} 
\frac{\e(s)^2}{t(s)}\bigg/ \frac{\tilde{\e}(s)^2}{\tilde{t}(s)} = 1.
\]
The curves $t = \tau \e^2$ (parametrized by $s \mapsto (\tau s^2, s)$), $\tau \geq 0$, provide 
representatives of each equivalence class except the one represented by the $t$ axis. There is a 
unique minimal $\calC^\infty$ structure on $Q_0$ for which the lifts of smooth functions from $Q$ 
and the parabolic polar coordinates $r = \sqrt{t + \e^2}$, $t/r^2$ and $\e^2/r^2$ are all smooth.
We label the faces of $Q_0$ as follows: $F$ is the new front face, and $L$ and $R$ are the left 
and right side faces (the lifts of $t=0$ and $\e=0$, respectively).   There is a smooth `blowdown'
map $\beta: Q_0 \longrightarrow Q$ defined in the obvious way.

It is usually more convenient to use projective rather than polar coordinates. There are two such
systems,
\[
(\tau,\e),\quad \tau = t/\e^2, \qquad \qquad \mbox{and}\quad (t,\eta),\quad \eta = \e/\sqrt{t},
\]
which are valid away from $R$ and $L$, respectively. Thus, for example, $\tau$ is an
`angular' variable which vanishes on $L$, and in this coordinate system $F = \{\e = 0\}$. 
 
Parabolic blowups are described in detail and much greater generality in \cite{tapsit}. 


\subsection{Polyhomogeneous conormal functions}
Let $M$ be a manifold with corners. A class of functions which is the natural replacement for (or at
least just as good as) the class of smooth functions is the class of polyhomogeneous conormal functions.
We refer to \cite{edge} for a detailed exposition, but review a few facts about these here. 

First recall the space $\calV_b$ of all smooth vector fields on $M$ which are tangent to all boundaries of $M$.
If $H_1, \ldots, H_k$ are boundary hypersurfaces of $M$ meeting at a corner of codimension $k$, with boundary 
defining functions $x_1, \ldots, x_k$, respectively, and local coordinates $y = (y_1, \ldots, y_{n-k})$ on the
corner, then  $\calV_b$ is spanned over $\calC^\infty(M)$ locally near this corner by 
$\{x_1\del_{x_1}, \ldots, x_k\del_{x_k}, \del_{y_1}, \ldots, \del_{y_{n-k}}\}$.  

A function (or distribution) $u$ is said to be conormal if it has stable regularity with respect to $\calV_b$. In 
other words, there exists a $k$-tuple of real numbers $\mu_1, \ldots, \mu_k$ so that
\[
V_1 \ldots V_\ell \, u \in x_1^{\mu_1}\ldots x_k^{\mu_k} 
L^\infty(M), \qquad \forall\, \ell\quad \mbox{and}\ \forall\, V_j \in \calV_b.
\]
(In particular, the $\mu_i$ are independent of $\ell$ and the $V_j$.)  Examples include monomials
$x_1^{s_1}\ldots x_k^{s_k}$ for  $s_j \in {\mathbb C}$, as well as products of arbitrary powers of
$|\log x_j|$. (This definition is slightly inaccurate since it omits the distributions supported at
the boundary, i.e.\ delta sections and their derivatives, which are also conormal, but suffices here.)
The special subclass with which we are interested consists of the functions with 
asymptotic expansions in terms of powers of the boundary defining functions and nonnnegative integer 
powers of the logs of these defining functions, with coefficients which
are smooth in all other variables. The expansions are formalized using the notion of an index set $I$.
This consists of a countable sequence of pairs $(\alpha,N) \in {\mathbb C} \times  \left\{ {\mathbb N} \cup \{0\} \right\}$ such that
for each $A \in \RR$, $\mbox{Re}\, \alpha > A$ for all but a finite number of these pairs.  Now, the conormal 
function $u$ has a polyhomogeneous expansion near a corner of codimension $k$ if there are $k$ index 
sets $I_1, \ldots, I_k$ so that
\[
u \sim \sum_{(\alpha_j,N_j) \in I_j} \sum_{\ell_j \leq N_j}  x_1^{\alpha_1}(\log x_1)^{\ell_1}
\ldots x_k^{\alpha_k}(\log x_k)^{\ell_k} a_{\alpha,\ell}(y),
\]
where each coefficient function $a_{\alpha,\ell}$ is $\calC^\infty$.  Note that since $u$ is
already assumed to be conormal, this expansion may be differentiated. 

The polyhomogeneous functions on $Q$ and $Q_0$ with which we shall be concerned are quite simple.
None of them will have log terms in their expansions, and the exponents are (not necessarily 
nonnegative) integers.  Thus, for example, near $L$ a polyhomogeneous function $u$ will have expansion 
in powers of $t$ with coefficients smooth in $\e$; near $F$ in terms of either of the projective 
coordinate systems, it has an expansion in powers of $\e$ with coefficients smooth in 
$\tau$, or equivalently, in powers of $t$ with coefficients smooth in $\eta$; near the corner
$L \cap F$ it will have an expansion in powers of $\tau$ and $\e$, with coefficients now simply numbers. 

The final point to describe here is that if $u$ is polyhomogeneous conormal on $Q$, then its lift
$\beta^* u$ to $Q_0$ is also polyhomogeneous conormal, and 
\[
u \sim \sum a_{jk} t^j \e^k \Longrightarrow \beta^* u \sim \sum a_{jk} (\tau \e^2)^j \e^k = \sum a_{jk} \tau^j \e^{2j+k}.
\]
On the other hand, if $w$ is polyhomogeneous on $Q_0$, then its pushforward to $Q$ is always conormal,
but rarely polyhomogeneous.

\subsection{Parametrix construction}
We conclude this section by reviewing a parametrix construction for the heat kernel, which is useful 
because it accurately captures the asymptotics of the true heat kernel as $t \searrow 0$. 
The construction here is slightly nonstandard, but is well suited for our calculations below. 

Let $M$ be a complete Riemannian manifold, possibly with boundary, and suppose that $M = M_1 \cup M_2$ 
where $M_1$ and $M_2$ are two manifolds with boundary with $M_1 \cap M_2 = \Sigma$ a hypersurface.  If
$M$ has boundary, assume that $\Sigma$ intersects $\del M$ transversely, and $M_1$ and $M_2$ are manifolds 
with corners of codimension two. Suppose further that $M_j$ lies in a slightly larger complete manifold 
$M_j'$, again possibly with boundary, such that for some neighbourhood $\calU$ of $\Sigma$, $M_j' \cap 
\calU = M \cap \calU$. 

Taking the heat kernels on each $M_j'$ as given, define
\[
\tilde{H}^M(t,z,z') = \sum_{j=1}^2 \chi_j(z) H^{M_j'}(t,z,z') \chi_j(z'),
\]
where $\chi_j$ is the characteristic function of $M_j$ in $M$.  In the more customary parametrix
construction, the $M_j$ are relatively open in $M,$ and $M_1 \cap M_2$ is also open; the $H^{M_j'}$
are pasted together using cutoff functions $\{\psi_j\}$ and $\{\tilde{\psi}_j\}$ with $\psi_1 + 
\psi_2 = 1$, where $\mbox{supp}\,\psi_j \subset \{\tilde{\psi}_j=1\}$, and $\mbox{supp}\,\tilde{\psi}_j
\subset M_j'$. We are using sharp (discontinuous) cutoffs rather than smooth ones, however, so that
we can identify certain asymptotic coefficients in the calculations to follow. 

\begin{lemma}
Let $H^M(t,z,z')$ denote the true heat kernel on $M$, and set 
\[
K(t,z) = \tilde{H}^M(t,z,z) - H^M(t,z,z).
\]
Then $K(t,z) = \calO(t^\infty)$ as $t \searrow 0$. 
\label{le:decomp1}
\end{lemma}
\begin{proof}
Rewrite
\begin{multline*}
\tilde{H}^M(t,z,z) = \chi_1(z)\left(H^{M_1'}(t,z,z) - H^M(t,z,z)\right)  \\ 
+ \chi_2(z)\left( H^{M_2'}(t,z,z) - H^M(t,z,z)\right) + H^M(t,z,z).
\end{multline*}
By assumption, $M_j'$ agrees with $M$ in a neighbourhood of $M_j$, so
that $H^{M_j'}(t,z,z) - H^M(t,z,z) = \calO(t^\infty)$ on the support of $\chi_j$
(remember that the small $t$ expansions of these operators are local), and
this proves the claim. 
\end{proof}


\section{Proofs of main theorems}
We have now assembled all the requisite facts and can proceed with the proofs of
the main theorems. 

As in the introduction, let $G(t,\e) = \Tr H^{\Omega_\e}$. If $\beta: Q_0 \to Q$ is the blowdown 
map, then let $\calG = \beta^* G$. We need to analyze the behaviour of $\calG$ near each of the 
faces and corners of $Q_0$, and for that we shall use the coordinates $(\tau,\e)$ introduced 
in \S 2.2. 

We shall make a simplifying assumption about the geometry in order to elucidate the proof.  For
each $i$, let $S_{\alpha_i}$ denote the sector in $\RR^2$ with opening angle $\alpha_i$. Choose
a smoothly bounded region $Z_i$ in the plane which coincides with $S_{\alpha_i}$ outside $B_{1/2}(0)$,  
and let $Z_i^{\e} = B_{1/\e}(0) \cap Z_i$. Then we assume that near each vertex $p_i$, the restriction
of the metric $g_\e$ to $B_1(p_i) \cap \Omega_\e$ is isometric to the dilation by the factor $\e$
of the region $Z_i^\e$, which obviously lies in the unit ball. The result remains true in the generality 
with which it was stated earlier, 
but the proof requires a few more technical steps which are both standard and not particularly 
germane to the main ideas here. Furthermore, for notational convenience only, we assume that there 
is only a single vertex $p$ and denote the corresponding smooth model region and sector by $Z$ 
and $S$, respectively. 

\medskip

\noindent{\em Proof of Theorem \ref{th:blowup}:} 
We first construct a particular family of parametrices for the heat kernel on $\Omega_\e$.
For any $0 \leq \e < \e_0$, decompose
\[
\Omega_\e = \Omega_{\e,1} \cup \Omega'
\]
where $\Omega_{\e,1} = \Omega_\e \cap B_1(p),$ and $\Omega' = \Omega_\e \setminus (\Omega_\e 
\cap B_1(p))$.  Note that $\Omega'$ is independent of $\e$. 
Lemma \ref{le:decomp1} shows that
\begin{equation}
H^{\Omega_\e}(t,z,z)  = \chi_1(z) H^{\e Z}(t,z,z) + \chi_2(z) H^{\Omega_0}(t,z,z) +  K(t,z),
\label{eq:hkoe}
\end{equation}
where $\chi_1$ is the characteristic function of $|z| \leq 1,$ $\chi_2 = 1-\chi_1$,
and $K$ is the error term from Lemma \ref{le:decomp1}, hence 
\[
G(t,\e) = \int_{|z| \leq 1} H^{\e Z}(t,z,z)\, dz + \int_{\Omega'} H^{\Omega_0}(t,z,z)\, dz 
+ \int_{\Omega_\e} K(t,z,z)\, dz.
\]
We denote the sum on the right side by $\mbox{I} + \mbox{II} + \mbox{III}$, and analyze the lifts of 
these terms successively.

By Proposition \ref{pr:dilate}, $H^{\e Z}(t,z,z') = \e^{-2} H^Z(t/\e^2, z/\e, z'/\e)$, so setting $z = z' = \e w$, 
we see that
\[
\beta^*\mbox{I} = \int_{|w| \leq 1/\e} H^Z(\tau, w, w)\, dw.
\]
This will be the principal term, and we defer its analysis for the moment.

Next, $\mbox{II}$ is independent of $\e$, and it is polyhomogeneous as $t \searrow 0$,
with expansion given by integrating the standard heat coefficients $a_j(z)$
over this restricted domain. Hence its lift to $Q_0$ is clearly polyhomogeneous.

Finally, by Lemma \ref{le:decomp1}, $\mbox{III}$ depends on $\e$ but decays rapidly 
in $t$ uniformly in $\e$. 

We now examine $\beta^*\mbox{I}$ more closely. Choose a smoothly bounded compact region $W$ which agrees
with $Z$ in $|w| \leq 2$, so that $Z = (W \cap B_1) \cup (S \setminus B_1)$. 
Using Lemma \ref{le:decomp1} again, write 
\begin{equation}
H^Z(t,z,z) = \chi_1(z) H^{W}(t,z,z) + \chi_2(z) H^S(t,z,z) + K_1(t,z)
\label{eq:hkZ}
\end{equation}
where $K_1$ is the corresponding error term. Then
\begin{multline*}
\beta^* \mbox{I} = \int_{|w| \leq 1} H^W (\tau,w,w)\, dw + \int_{1 \leq |w| \leq 1/\e} H^S(\tau,w,w)\, dw \\
+ \int_{|w| \leq 1/\e} K_1(\tau,w,w)\, dw,
\end{multline*}
which we write as $\mbox{I}' + \mbox{II}' + \mbox{III}'$.  

We first prove polyhomogeneity of these terms away from the right face R of $Q_0$.
The term $\mbox{I}'$ has an expansion as $\tau \searrow 0$ and is independent of $\e$,  
so $\beta^*\mbox{I}'$ is certainly polyhomogeneous in this region. By Lemma \ref{le:decomp1} 
again, $K_1$ decreases rapidly as $\tau \to 0$, so this term is also polyhomogeneous there.
Note too that by the explicit form of the error term in the proof of that lemma, and using
the dilation properties of $H^Z$ and $H^S$ again, $K_1(\tau,z) = \calO(|z|^{-\infty})$ uniformly
for $\tau$ in any bounded set, so its integral over $|z| \leq 1/\e$ is also bounded independently
of $\e$. 

To analyze the remaining term, set
\[
D(R) := \int_{|w| \leq R} H^S(1, w,w)\, dw. 
\]
By Proposition \ref{pr:dilate}, $\mbox{II}'(\e,\tau) = D(1/\e \sqrt{\tau}) - D(1/\sqrt{\tau})$,
so it will suffice to show that $D$ has an expansion in powers of $1/R$ as $R \to \infty$. 
For this, we appeal to a calculation by van den Berg and Srisatkunarajah \cite{BS}, who prove that
\begin{equation}
D(R) = \frac{\alpha R^2}{8 \pi }  - \frac{R^2}{2\pi } \int_0^1 e^{-R^2 y^2} \sqrt{1-y^2}\, dy + 
\frac{\pi^2 - \alpha^2}{24 \pi \alpha} + \calO(e^{-cR^2})
\label{eq:BS}
\end{equation}
for some $c > 0$ independent of $R$. Only the polyhomogeneous structure of the second term on the right 
is not completely obvious.  For that, we may as well replace the upper limit of integration by $1/2$ since the 
integral from $1/2$ to $1$ decreases exponentially in $R$.  Using the Taylor series for $\sqrt{1-y^2}$ 
at $y=0$, we find that
\[
\frac{R^2}{2\pi} \int_0^{1/2} e^{-R^2 y^2} (1 - \frac12 y^2 - \frac14 y^4 - \ldots)\, dy
\sim \frac{R}{4\sqrt{\pi}} - \frac{1}{16\sqrt{\pi}R}  + \calO(R^{-3}),
\]
and this completes the proof of polyhomogeneity of $\beta^*\mbox{I}$ for $\tau$ in any bounded set.

To finish the proof, we must analyze the behaviour of $\beta^*\mbox{I}$ as $\tau \nearrow \infty$. Switch
to the coordinates $t,\eta$, so $\e = \eta \sqrt{t}$ and $\tau = \eta^{-2}$. 
It is now more convenient to use the standard representation of the heat kernel in terms of the resolvent:
\begin{equation}
H^Z = \int_\Gamma e^{-\tau\lambda} R_Z(\lambda)\, d\lambda, \qquad \mbox{where}\quad
R_Z(\lambda) = (-\Delta_Z - \lambda)^{-1}.
\label{eq:heatres}
\end{equation}
Here $\Gamma$ is a path surrounding the spectrum of $\Delta_Z$, for example, the two half-lines
$\mbox{Im}\, \lambda = \pm(\alpha \mbox{Re}\, \lambda + \beta)$, $\alpha, \beta > 0$, 
joined by the half-circle $|\lambda| = \beta$, $\mbox{Re}\,\lambda \leq 0$, traversed in the 
counterclockwise direction. 
In general it is a subtle matter to deduce the fact that $H^Z$ has an expansion in powers of $1/\tau$
at large times since this depends on the fine structure of the resolvent near the threshold $\lambda = 0$.
However, in this case we already have sufficient information about the heat kernel on $S$ that this
is not hard. Choose a partition of unity $\{\psi_1, \psi_2\}$ on $Z$ such that $\psi_1 = 1$ in
$|z| \leq 3/4$ and $\psi_2 = 1$ in $|z| \geq 5/4$, and that both $Z \cap W$ and $Z \cap S$
contain the region $Z \cap \{3/4 \leq |z| \leq 5/4\}$. Choose other cutoff functions $\tilde{\psi}_j$
such that $\tilde{\psi}_j = 1$ on $\mbox{supp}\,\psi_j$. Let $R_W$ and $R_S$ denote the resolvents for
$\Delta_W$ and $\Delta_S$, with Dirichlet boundary conditions, and define the parametrix 
\[
\tilde{R}_Z(\lambda) = \tilde{\psi}_1 R_W(\lambda) \psi_1 + \tilde{\psi}_2 R_S(\lambda) \psi_2.
\]
Then
\[
(\Delta_Z - \lambda) \tilde{R}_Z(\lambda) = I + [\Delta_Z,\tilde{\psi}_1]R_W(\lambda) \psi_1 + 
[\Delta_Z, \tilde{\psi}_2] R_S(\lambda) \psi_2 :=  I + E(\lambda).
\]
Since the singular supports of both $R_W$ and $R_S$ are on the diagonal, and the support of 
$[\Delta,\tilde{\psi}_j]$ is disjoint from that of $\psi_j$, we see that $E(\lambda)$ is a holomorphic
family of operators (for $\lambda \in {\mathbb C} \setminus \RR^+$) which maps $L^2(Z)$ into 
$\calC^\infty_0(Z)$. Also, since $-\Delta_Z - \lambda$ is invertible for $\lambda$ in this region, 
$I + E(\lambda)$ is also invertible there. We write its inverse as $I + F(\lambda)$, so that
\begin{equation}
R_Z(\lambda) = \tilde{R}_Z(\lambda) + \tilde{R}_Z(\lambda)F(\lambda).
\label{eq:resparam}
\end{equation}
The relationships $(I + E(\lambda))(I + F(\lambda)) = (I + F(\lambda)) (I + E(\lambda)) = I$ imply that 
\[
F(\lambda) = -E(\lambda) + E^2(\lambda) + E(\lambda) F(\lambda) E(\lambda),
\]
hence $F(\lambda)$ is also smoothing and maps $L^2(Z)$ into $\calC^\infty_0(Z)$; the second term on the right 
in (\ref{eq:resparam}) has the same mapping properties.  

Finally, the form of the expansion of $F(\lambda)$ for $\lambda$ near $0$ (away from the positive real axis) 
is the precisely the same as that of $E(\lambda)$, which in turn is the same as that of $\tilde{R}_Z(\lambda)$,
and hence finally as that of $R_S(\lambda)$. 

The decay of each of term as $|\lambda| \to \infty$ is straightforward, so we can write
\begin{multline*}
\int_{|z| \leq 1/\eta \sqrt{t}} H^Z(\eta^{-2},z,z)\, dz = 
\int_{|z| \leq 1/\eta \sqrt{t}} \left(\int_\Gamma e^{-\lambda/\eta^2} R_Z(\lambda)\, d\lambda\right) (z,z)\, dz \\
= \int_{|z| \leq 2} \left(\int_\Gamma e^{-\lambda/\eta^2} \psi_1 R_W(\lambda)\, d\lambda\right)(z,z)\, dz  \\ 
+ \int_{1/2 \leq |z| \leq 1/\eta \sqrt{t}}\left(\int_\Gamma e^{-\lambda/\eta^2} \psi_2 R_S(\lambda)\, d\lambda\right)(z,z)\, dz \\ +
\int_{|z| \leq 1/\eta \sqrt{t}} \left(\int_{\Gamma} e^{-\lambda/\eta^2} \tilde{R}_Z(\lambda)F(\lambda)\, d\lambda\right) (z,z)\, dz.
\end{multline*}
The inner integrand in the first term on the right extends holomorphically to a neighbourhood of $\lambda = 0$, 
so the contour can be moved to lie entirely in the right half-plane, which shows that this term decreases
exponentially in $1/\eta$. The second term is polyhomogeneous by the explicit analysis of the function $D(R)$
above. The fact that the final term has an expansion follows from the existence of asymptotics of
$F(\lambda)$ for $\lambda$ near $0$. This completes the proof of the polyhomogeneity of $\calG$ on $Q_0$. 
\hfill $\Box$

\medskip

\noindent{\em Proof of Theorem \eqref{th:fp}}
This consists of examining the terms in the expansion of $\calG$ at the various boundary faces.

First, at L, away from F we may use the variables $(t,\e)$, and 
\[
G(t,\e) \sim \sum_{j=0}^\infty a_j(\e) t^{-1 + j/2}.
\]
Near L $\cap$ F, we substitute $t = \e^2 \tau$ to get 
\begin{equation} \label{eq:coeff}
\calG(\tau,\e) \sim \sum_{j=0}^\infty a_j(\e) \tau^{-1 + j/2} \e^{-2 + j}.
\end{equation}
The coefficients $a_j(\e)$ are polyhomogeneous as $\e \to 0$ by Theorem \ref{th:blowup}.   

At R, away from $t=0$, 
\[
G(t,\e) \sim \sum_{j=0}^\infty B_j(t) \e^j;
\]
here $B_0(t) = \Tr H^{\Omega_0}$. Near F $\cap$ R we use the coordinates $t$ and $\eta = \e/\sqrt{t}$ to compute 
\[
\calG(t,\eta) \sim \sum_{j=0}^\infty B_j(t) \eta^j t^{j/2}.
\]
Again, the coefficients $B_j(t)$ are polyhomogeneous in $t$. 

Finally, near F, we use the coordinates $(\tau,\e)$, so the expansion is in powers of $\e,$ and by (\ref{eq:coeff}) it is 
\[
\calG(\tau,\e) \sim \sum_{j=0}^\infty C_j(\tau) \e^{-2 + j}.
\]
We shall identify the coefficients $C_0$, $C_1$ and $C_2$. 

By our analysis of the terms $\mbox{I}'$, $\mbox{II}'$, $\mbox{III}'$, $\mbox{II}$ and $\mbox{III}$, 
we see that only $\mbox{II}'$ and $\mbox{II}$ contribute to the coefficients of $\e^{-2}$ and $\e^{-1}$.
Substituting directly from the expansions of these two terms (using the McKean-Singer asymptotics on
$\Omega'$ for $\mbox{II}$ and the first terms in the expansion of $D(1/\e \sqrt{\tau})$ for $\mbox{II}'$),
and then using the definition of the finite part at $\e = 0$ of $\mbox{I}$, we have
\begin{multline*}
\calG(\tau,\e) \sim \frac{1}{\e^{2}\tau}\left(\frac{|\Omega'|}{4\pi} + \frac{\alpha}{8\pi}\right) 
- \frac{1}{\e \tau^{1/2}}\left(\frac{|\partial \Omega'|}{8\sqrt{\pi}} + \frac{1}{4\sqrt{\pi}}\right)  \\
+ \frac{1}{12 \pi} \left(\int_{\Omega'} K\, dA + \int_{\del \Omega'} \kappa\, ds \right)  + 
\underset{{\e=0}}{\mathrm{f.p.}} \int_Z H^Z(\tau,w,w)\, dw+ \calO(\e).
\end{multline*}
In other words,   
$$C_0 (\tau) = \frac{1}{\tau} \left( \frac{|\Omega'|}{4 \pi} + \frac{\alpha}{8 \pi} \right),$$
$$C_1 (\tau) = - \frac{1}{\sqrt{\tau}} \left( \frac{|\partial \Omega'|}{8 \sqrt{\pi}} + \frac{1}{4 \sqrt{\pi}} \right),$$
and
$$C_2 (\tau) = \frac{1}{12 \pi} \left( \int_{\Omega'} K dA + \int_{\partial \Omega'} \kappa ds \right) + \underset{{\e=0}}{\mathrm{f.p.}} \int_Z H^Z(\tau,w,w)\, dw.$$
 
This simplifies using the following observations: first, the area of a circular sector of opening $\alpha$ and radius
$1$, i.e.\ $|\Omega_0 \cap B_1|$, equals $\alpha/2$, so the coefficient of $\e^{-2}\tau^{-1}$ is just $|\Omega_0|/4\pi$; 
similarly, the sides of this circular sector are straight lines, so $|\del \Omega_0 \cap B_1| = 2$, which means
that the next coefficient is $-|\del \Omega_0|/8\sqrt{\pi}$; finally, since $g_0$ is flat in $\Omega_0 \cap B_1$,
$K \equiv 0$ there, so using that the contribution from `turning the corner' at $p$ in the boundary integral is $\pi-\alpha$,
we find that
\[
\int_{\Omega'} K\, dA + \int_{\del \Omega'} \kappa \, ds = 2\pi \chi(\Omega_0) - (\pi-\alpha). 
\]
This means that 
\begin{equation}
C_2(\tau) = \underset{{\e=0}}{\mathrm{f.p.}} \int_Z H^Z(\tau,w,w)\, dw + \frac{1}{6}\chi(\Omega_0)-\frac{\pi-\alpha}{12\pi}.
\label{eq:expC2}
\end{equation} 

We conclude by calculating its behaviour for small and large $\tau$.  Using the small $\tau$ asymptotics,
we see that
\begin{multline*}
\int_{|w| < 1/\e} H^Z(\tau,w,w)\, dw \sim \frac{|Z \cap B_{1/\e}|}{4\pi} \tau^{-1}  \\ 
- \frac{|\del Z \cap B_{1/\e}|}{8\sqrt{\pi}}\tau^{-1/2} + \frac{1}{12 \pi}\int_{\del Z} \kappa \, ds + \calO(\e \tau^{1/2});
\end{multline*}
hence the finite part of this integral is equal (up to the factor $12\pi$) to the integral of curvature on the 
boundary of $Z$, which is the total turning angle $\pi-\alpha$, so finally, the limit of $C_2$ as $\tau \to 0$
is $\chi(\Omega_0)/6$, as claimed.

Finally, we use the dilation one more time to calculate that 
\[
\int_{|w| \leq 1/\e} H^Z(\tau,w,w)\, dw = \int_{|w| \leq 1/\e \sqrt{\tau}} H^{Z/\sqrt{\tau}}(1,w,w)\, dw.
\]
Noting that $\e \sqrt{\tau} = \sqrt{t}$, and since $Z/\sqrt{\tau}$ converges to the sector $S$ as $\tau \to \infty$, 
we can use the expansion (\ref{eq:BS}) to see that the finite part is indeed $(\pi^2 - \alpha^2)/24 \pi \alpha$.  
Therefore, in general, with an arbitrary number of vertices,
\[
\lim_{\tau \to \infty} C_2(\tau) = \frac{\chi(\Omega_0)}{6} + \sum_{j=1}^k \frac{\pi^2 - \alpha_j^2}{24\pi \alpha_j}
- \frac{1}{12\pi} \sum_{j=1}^k (\pi-\alpha_j);
\]
in particular, if $\Omega_0$ is a polygon, its Euler characteristic is $1$, so the first and third terms cancel. 

This completes the proof. \qed

\section{Neumann boundary conditions}
We now briefly discuss the minor modifications needed to prove the analogues of Theorems \ref{th:blowup} and \ref{th:fp}
if Neumann conditions are used instead of Dirichlet conditions. 

A cursory inspection of the proof shows that the only real issue is to find an analogue of the van den Berg-Srisatkunarajah 
formula \eqref{eq:BS} in this setting.  This does not seem to appear explicitly in the literature, but fortunately, a recent paper 
by Kokotov \cite{Kok} contains the corresponding formula for the complete cone $C_{2\alpha}$ of angle $2\alpha$.
Let $H^C$ denote the heat kernel on this cone. Then by Proposition 1 in \cite{Kok}, there exists $c > 0$ such that 
for every $R > 0$, 
\begin{equation}
\int_{|z| \leq R} H^{C}(1,z,z)\, dz = \frac{\alpha R^2}{4\pi} + \frac{1}{12}\left(\frac{4\pi^2 - (2\alpha)^2}{2\pi (2\alpha)}\right)
+ \calO(e^{-cR^2}).
\label{eq:Kok}
\end{equation}
This formula is stated in \cite{Kok} for fixed radius $R$ and for the heat kernel at time $t$ as $t \to 0$, but because of the usual
scaling properties, it holds equally well for fixed $t$, say $t=1$, and as the radius $R \to \infty$; indeed, the quantity on the left 
depends only on the ratio $R/t^2$.  The coefficients in this expansion have been written in a nonreduced form in order to 
emphasize the dependence on the angle $2\alpha$. 

We now observe that the cone $C_{2\alpha}$ is the union of two copies of the sector $S_\alpha$ with the boundary rays
identified. Alternately, let $\tau$ be the obvious reflection on the cone $C_{2\alpha}$; then a region isometric to the sector 
$S_\alpha$ is a fundamental domain for this action and its image $\tau(S_\alpha)$ is the other half of the cone.  
In any case, using this, the formula for the Neumann heat kernel follows directly from \eqref{eq:BS} and \eqref{eq:Kok}. 
Indeed, let $L^2(C_{2\alpha}) = L^2_+ \oplus L^2_-$ be the decomposition into functions which are even and odd with respect 
to $\tau$. If $u \in H^2(C_{2\alpha}) \cap L^2_+$, then $u$ has vanishing normal derivative at $\del S_\alpha$, while if 
$u \in H^1(C_{2\alpha}) \cap L^2_-$ then $u$ vanishes at $\del S_\alpha$. Since the Laplacian commutes with $\tau$, 
the heat kernel has a $2$-by-$2$ block decomposition: the upper left and lower right on-diagonal blocks are canonically 
identified with the Neumann and Dirichlet heat kernels of $S_\alpha$, and we denote these by $H^{S}_{\mathrm{N}}$ and 
$H^{S}_{\mathrm{D}}$, respectively. Therefore, 
\[
\int_{|z| \leq R} \left(H^{S}_{\mathrm{D}}(1,z,z) + H^{S}_{\mathrm{N}}(1,z,z)\right)\, dz =
\int_{|z| \leq R} H^{C}(1,z,z)\, dz,
\]
whence
\begin{equation}
\int_{|z| \leq R} H^{S}_{\mathrm{N}}(1,z,z) \, dz \sim \frac{\alpha R^2}{4\pi} + \frac{R}{4\sqrt{\pi}} + 
\frac{\pi^2 - \alpha^2}{24 \pi \alpha} + \ldots
\label{eq:Neu}
\end{equation}
In other words, in this asymptotic formula, only the signs of the odd powers of $R$ are reversed from those in the corresponding
formula for the Dirichlet heat kernel.

It is now a simple matter to track through the various arguments in this paper to obtain that $\calG_{\mathrm{N}}(\tau,\e)$,
the pullback to $Q_0$ of the trace of the heat kernel for the Laplacian with Neumann boundary conditions on $\Omega_\e$,
is polyhomogeneous and has the expansion
\begin{multline*}
\calG_{\mathrm{N}}(\tau,\e) \sim \frac{1}{\e^2 \tau} \frac{|\Omega_0|}{4\pi} + \frac{1}{\e \tau^{1/2}}
\frac{|\del \Omega_0|}{8\sqrt{\pi}} +  \\
\frac{1}{12 \pi} \left(\int_{\Omega'} K\, dA + \int_{\del \Omega'} \kappa\, ds \right)  + 
\underset{{\e=0}}{\mathrm{f.p.}} \int_{|w| < 1/\e} H^Z_{\mathrm{N}}(\tau,w,w)\, dw+ \calO(\e).
\end{multline*}
In particular, the coefficient $C_2(\tau)$ of $\e^0$ is exactly the same as in the Dirichlet case. 
We leave the straightforward details to the reader. 

\section{Higher dimensions and other generalizations}
We have focused in this paper on two-dimensional domains in order to emphasize the simplicity of the arguments
and to take advantage of the explicit nature of the various formul\ae.  There are various analogues of these
results in higher dimensions, which we now describe briefly. These generalizations should have some interesting
applications, which will be developed elsewhere.  

The simplest generalization is to consider a family of Riemannian metrics $g_\e$ on a compact manifold $M$ such
that $(M,g_\e)$ degenerates to a space $(M_0,g_0)$ which has isolated conic singularities. We assume that this degeneration
is modelled on the rescalings of a complete asymptotically conic space $(Z,g_Z)$, i.e.\ such that suitable neighbourhoods
of $(M,g_\e)$ are (asymptotically equivalent to) rescalings of truncations of $(Z,g_Z)$. The behaviour of the entire heat kernel 
for this type of degeneration was studied in great detail in \cite{R0}, \cite{R1}. The analysis in those papers is much
more general, but considerably more intricate, than what is done here, but one consequence of those results is the
fact that the trace of the heat kernel for $\Delta_{g_\e}$ lifts to $Q_0$ to be polyhomogeneous, and the terms in its
expansions at the various faces can be determined explicitly. In particular, the coefficient $C_{n/2}(\tau)$ of $\e^0$ at the front
face of $Q_0$ is equal to the sum
\[
\int_{M_0} q_{n/2}\, dV + \underset{{\e=0}}{\mathrm{f.p.}} \int_{Z_\e} H^Z(\tau,w,w)\, dw,
\]
where $q_{n/2}$ is the standard heat invariant integrand for the metric $g_0$. It is unlikely, however, 
that one could find an explicit formula for the limit as $\tau \to \infty$ of this regularized trace, except in special cases. 

There should be a similar generalization of these ideas to the setting of resolution blowups of iterated edge spaces 
(or smoothly stratified spaces), introduced in \cite{evian}. This is likely to present a greater challenge, but one of 
the motivations for exploring the rescaling methods used here is to find way to circumvent the machinery of 
\cite{R0}, \cite{R1} if one is only interested in heat traces rather than more extensive information about the heat kernel. 

A special and very interesting case would be to find an analogue of Theorem \ref{th:fp} for smoothings of Euclidean polyhedra 
in arbitrary dimension.  The description of a family of `self-similar' smoothings of an arbitrary polyhedron is not difficult 
and follows the scheme presented in \cite{evian} closely. However, in order to make this formula explicit, one would need
an analogue of \eqref{eq:BS} or \eqref{eq:Kok} for higher dimensional polyhedral sectors, which does not seem to
be available. (Analogous results {\it are} known for other spectral invariants, however, see \cite{Fe} and \cite{Ch}.)

\end{document}